\documentclass[11pt,leqno]{amsart}
\usepackage{latexsym,amssymb}
\usepackage{hyperref}
\usepackage{amsmath,amsthm}
\usepackage{amsfonts}
\usepackage[american]{babel}
\usepackage{mathrsfs}
\usepackage{psfrag}

\numberwithin{equation}{section}
\newtheorem{theorem}{Theorem}[section]
\newtheorem{proposition}[theorem]{Proposition}

\newtheorem{lemma}[theorem]{Lemma}

\newtheorem{definition}[theorem]{Definition}

\newcommand{\rn}{\mathbb{R}^n}
\newcommand{\N}{\mathbb{N}}
\newcommand{\bk}{\bar{k}}
\newcommand{\bxi}{\bar{\xi}}
\newcommand{\br}{\bar{r}}
\newcommand{\bx}{\bar{x}}

\begin{document}

\title[Optimal concavity for torsion]{Optimal concavity of the torsion function}


\author{A. Henrot, C. Nitsch, P. Salani, C. Trombetti}


\date{}


\begin{abstract}
In this short note we consider an unconventional overdetermined problem for the torsion function:
let $n\geq 2$ and $\Omega$ be a bounded open set in $\rn$ whose torsion function $u$  (i.e. the solution to 
$\Delta u=-1$ in $\Omega$, vanishing on $\partial\Omega$) satisfies the following property:
$\sqrt{M-u(x)}$ is convex, where $M=\max\{u(x)\,:\,x\in\overline\Omega\}$. 
Then $\Omega$ is an ellipsoid. 
\end{abstract}

\maketitle

\section{Introduction}
When studying a (well posed) Dirichlet problem, a natural question is whether and how some relevant geometric property of the underlying domain influences the solution. A deeply investigated situation is when the domain is convex and the involved equation is elliptic. A classical result in this framework is the following (see Makar-Limanov \cite{ML} in the planar case, \cite{kaw, ken} for $n>2$).
\medskip

\begin{proposition}\label{prop0}
Let $n\geq 2$, $\Omega$ be a bounded open set in $\mathbf{R}^n$ and let $u$ solve
\begin{equation}\label{pbt}
\left\{
\begin{array}{ll}
\Delta u=-1 & \mathrm{in}\,\,\Omega \\
u=0 & \mathrm{on}\,\,\partial \Omega
\end{array}
\right.
\end{equation}
If $\Omega$ is convex, then $u$ is $(1/2)$-concave, i.e. $\sqrt{u}$ is a concave function.
\end{proposition}
We recall that the solution to problem \eqref{pbt} is  called the {\em torsion function} of $\Omega$, since the {\em torsional rigidity} of $\tau(\Omega)$ is defined by
as follows:
\begin{equation}
\label{capa}
\tau(\Omega)^{-1}=\min\left\{\frac{\int_{\Omega}|\nabla v|^2\,dx}{\left(\int_\Omega v\,dx\right)^2}\,:\,
v\in W^{1,2}_0(\Omega), v \not\equiv 0 \right\}\,;
\end{equation}
 the above minimum is  achieved by the solution $u$ to \eqref{pbt} and it holds $\tau(\Omega)=\displaystyle\int_\Omega u\,dx$.
\medskip

Nowadays there are several methods to prove general results like Proposition \ref{prop0} (see for instance \cite{kaw}, \cite{Caffarelli-Spruck}, \cite{kor2}), but not so much has been done to investigate the optimality of them.
When $\Omega$ is a ball, say $\Omega=B({\bx},R)=\{x\in\rn\:\,|x-{\bx}|<R\}$, the solution to \eqref{pbt} is 
$$
u_B(x)=\frac{R^2-|x-\bx|^2}{2n}\,,
$$
and it is {\em concave} (which is a stronger property than $(1/2)$-concavity). More generally, the same happens for every ellipsoid 
$E=\{x\in\rn\,:\,\sum_{i=1}^na_i(x_i-\bx_i)^2<R^2\}$, with $a_i>0$ for $i=1,\dots,n$, $\sum_{i=1}^na_i=n$; in this case the solution is 
$$
u_E(x)=\frac{R^2-\sum_{i=1}^na_i(x_i-\bx_i)^2}{2n}\,,
$$
and it is and it is {\em concave}.
One can wonder whether this concavity property of the torsion function characterizes balls or ellipsoids
(as it is the case for the Newtonian potential as shown in \cite{Salani}, see below). 
Actually the answer is negative since  there are convex domains whose  torsion function is concave, for instance:
\begin{itemize}
\item small perturbations of balls or ellipsoids (since $u_B$ and $u_E$ are uniformly concave).
\item the torsion function of the equilateral triangle of vertices $(-2,0),\,(1,\sqrt{3}),\,
(1,-\sqrt{3})$ is given by $u_T(x,y)=(4-3y^2+3xy^2-3x^2-x^3)/12$. Since the trace of the Hessian matrix 
of $u_T$ is $-1$ and its determinant is $(1-x^2-y^2)/4$, then  any (convex) level  set $u_T=c$ included into the unit disk has a convex torsion function ($u_T-c$).
\item more generally, for any convex domain $\Omega$, any (convex) level set sufficiently close to the maximum of
$u$ can provide a domain where the torsion function is concave.
\end{itemize}
Since the power concavity has a monotonicity property (namely, $u^\alpha$ concave $\Rightarrow$ $u^\beta$ concave
for $\beta\leq\alpha$), we can introduce the {\it torsional concavity exponent} of a convex domain $\Omega$ as the
number $\alpha^*(\Omega)$ defined as
$$\alpha^*(\Omega)=\sup\{\alpha>0,\ \mbox{such that } u^\alpha\ \mbox{is concave}\}$$
where $u$ is the torsion function of $\Omega$. Then we have the following property which shows that the ellipsoids and
many other domains maximize this quantity $\alpha^*(\Omega)$. Note that the same question has been raised by P. Lindqvist in \cite{lindqvist} about the first eigenfunction of the Dirichlet-Laplacian
and this question of optimality of the ball seems to be still open for the eigenfunction.
\begin{proposition}
For any   bounded convex open set, we have $\frac{1}{2}\,\leq \alpha^*(\Omega) \leq 1$.
\end{proposition}
\begin{proof}
The first inequality comes from Proposition \ref{prop0}. Let us prove the second inequality. Let $u$ be the torsion
function of the domain $\Omega$ and let $\alpha>1$ be fixed, for $x \in \Omega$ it results
 $\Delta (u^\alpha)=\alpha u^{\alpha-2}[(\alpha-1)|\nabla u|^2 -u]=G(x)$.  Observing that $\displaystyle\int_{u= \epsilon}| Du| = \textrm{Vol}\{u>\epsilon\}$ one can deduce that $\displaystyle\int_{u= \epsilon} \Delta (u^\alpha) \ge \alpha \epsilon^{\alpha -2} \left[ (\alpha -1) \dfrac{\textrm{Vol}^2\{u>\epsilon\}}{\textrm{Per}\{u>\epsilon\}} -\epsilon \ \textrm{Per}\{u>\epsilon\} \right],$ which is positive for $\epsilon$ small enough since $\textrm{Vol}\{u>\epsilon\}\to\textrm{Vol}(\Omega)$ and $\textrm{Per}\{u>\epsilon\}\to\textrm{Per}(\Omega)$ as $\epsilon \to 0$.
Therefore there exists   a point $x_0\in  \Omega$ where  $G$ is   positive, hence  $u^\alpha$ cannot be concave in $\Omega$.
\end{proof}

\medskip
In order to get a property which characterizes balls and ellipsoids, we  introduce 
the property (A), defined as follows.
\begin{definition}\label{propertyA}
Let $\Omega$ be a bounded convex open set in $\rn$. We say that a function $v\in C(\overline\Omega)$ satisfies  property (A) in $\Omega$ if
$$
(A)\qquad w(x)=\sqrt{M-v(x)}\quad\text{is   convex  in $\Omega$}\,,
$$
where $M=\max_{\overline\Omega}v$.
\end{definition}
It is easily seen that if a function $v$ satisfies property (A) is concave and also  $(1/2)$-concave. 
%
Then one can suspect that the result by Makar-Limanov and Korevaar may be improved and could, for instance, guess that property (A) is satisfied by the solution $u$  to problem \eqref{pbt} as soon as $\Omega$ is convex.

We will prove that this is not true and that property (A) is "sharp" for $u_E$, in the sense that it characterizes ellipsoids. Precisely our main result is the following.
\begin{theorem}\label{mainthm}
Let $\Omega$ be a bounded  open set  and let $u\in C^2(\Omega)\cap C(\overline\Omega)$ be the solution to \eqref{pbt}.
If $u$ satisfies property (A), then $\Omega$ is an ellipsoid and $u=u_E$.
\end{theorem}

As far as we know this is just the second step in the direction of investigating sharpness of concavity properties of solutions to elliptic equation, a first step being done by one of the authors in \cite{Salani}, where the following is proved: 
{\em let $n\geq 3$, $\Omega$ be a compact convex subset of $\rn$ and $u$ be  the Newtonian potential of $\Omega$, that is the solution to
\begin{equation}\label{pbc}
\left\{\begin{array}{ll}
\Delta u=0\quad&\text{in }\rn\setminus\overline\Omega\\
u=1\quad&\text{in }\overline\Omega\\
u\to0\quad&\text{as }|x|\to+\infty\,;
\end{array}\right.
\end{equation}
if $u^{2/(2-n)}$ is convex, then $\Omega$ is a ball.}

Notice that both the latter result and Theorem \ref{mainthm} can be regarded as (unconventional) overdetermined problems. 
In general, an overdetermined problem is a Dirichlet problem coupled with some extra condition and the prototypal one is the Serrin problem, where \eqref{pbt} is coupled with the following Neumann condition:
\begin{equation}\label{serrin}
|\nabla u|=\text{constant}\quad\text{on }\partial\Omega\,.
\end{equation}
In a seminal paper \cite{Serrin}, Serrin proved that a solution to \eqref{pbt} satisfying \eqref{serrin} exists if and only if $\Omega$ is a ball.
The literature about overdetermined problems is quite large, but
usually the extra condition imposed to the involved Dirichlet problem regards the normal derivative of the solution on the boundary of the domain, like in \cite{Serrin}, and the solution is given by the ball. Recently different conditions have been considered, like for instance in \cite{CM, CMS, CMV, CuLa, CuLa2, MaSa1, MaSa2, Sc, Sh}. More particularly, in \cite{Salani} the overdetermination is given by the convexity of $u^{2/(2-n)}$; here, in a similar spirit, the overdetermination in Theorem \ref{mainthm} is given by property (A). Again in connection with Theorem \ref{mainthm}, we also recall that overdetermined problems where the solution is affine invariant and it is given by ellipsoids are considered in \cite{BGNT, ES, HPh} etc.

\section{Proof of Theorem \ref{mainthm}.}
Throughout, $u\in C^2(\Omega)\cap C(\overline\Omega)$ denotes the solution to \eqref{pbt} and 
$$
M=\max_{\overline\Omega}u\,.
$$
Notice that the maximum principle gives
$$
0< u\leq M\quad\text{ in }\Omega\,.
$$
Without loss of generality (up to a translation), we can assume $B_\rho(0)\subset\Omega$ for some $\rho>0$ and 
$$u(0)=M\,.
$$
Then $\nabla u(0)=0$. Furthermore, up to a rotation, we can also assume
$$
D^ 2u(0)=\left(\begin{array}{cccc}
-\lambda_1&\dots&0\\
0&\ddots&0\\
0&\dots&-\lambda_n\end{array}\right)\,,
$$
with $\lambda_i\geq 0$ for $i=1,\dots,n$ and  
$$\sum_{i=1}^n\lambda_i=1\,,$$ 
thanks to the equation in \eqref{pbt}.
\medskip

Let $v(x)=M-u(x)$, then
\begin{equation}\label{eqv}
\left\{
\begin{array}{ll}
\Delta v=1 & \mathrm{in}\,\,\Omega \\
v=M & \mathrm{on}\,\,\partial \Omega\,.
\end{array}
\right.
\end{equation}
Moreover
$$
0\leq v<M\quad\mathrm{in}\,\,\Omega\,.
$$
and
\begin{equation}\label{v0}
v(0)=0\,,\quad\nabla v(0)=0\,,\quad D^ 2v(0)=\left(\begin{array}{cccc}
\lambda_1&\dots&0\\
0&\ddots&0\\
0&\dots&\lambda_n\end{array}\right)\,.
\end{equation}

Then we can write
$$
v(x)=\frac12\sum_{i=1}^n\lambda_ix_i^2+z(x)\,,
$$
where $z$ is a harmonic function in $\Omega$, such that
\begin{equation}\label{z0}
z(0)=0\,,\quad\nabla z(0)=0\,,\quad D^ 2z(0)=0\,.
\end{equation}

Theorem \eqref{mainthm} will be proved once we prove the following lemma.
\begin{lemma}\label{mainlemma}
Let $\lambda_1,\dots,\lambda_n\geq0$ such that $\sum\lambda_i=1$ and let $z$ be a harmonic function in a neighborhood $B_\rho(0)$ of the origin, satisfying \eqref{z0}.
If
$$
w(x)=\sqrt{\frac12\sum_{i=1}^n\lambda_ix_i^2+z(x)}
$$
is a convex function in $B_\rho(0)$, then $z\equiv 0$.
\end{lemma}
\begin{proof}
Let $r=|x|$ and $\xi=x/r\in S^{n-1}$. Then we can write
$$
z(x)=z(r,\xi)=\sum_{k=0}^\infty r^kz_k(\xi)
$$
with
$$z_k(\xi)=\sum_{j=1}^{N(k,n)}a_{k,j}Y_{k,j}(\xi)\,,
$$
where $a_{k,j}$ are suitable coefficients,
$$
N(k,n)=\frac{(2k+n-2)(k+l-3)!}{(n-2)!k!}
$$
and $Y_k^j$, $j=1,\dots,N(k,n)$ is an orthonormal basis of spherical harmonics of degree $k$ in dimension $n$.
We recall that the spherical harmonic $Y_k^j$ is a solution to $$-\Delta_\xi Y=k(k+n-2)Y$$ (where $-\Delta_\xi$ denotes the spherical Laplacian), whence
we have that the function $\tilde{z}_k(x)=r^kz_k(\xi)$ is harmonic in $\rn$ for every $k$. Then we notice that, since $\tilde{z}_k(0)=0$ we have that $\int_{|x|=\rho}\tilde{z}_k d\sigma=0$ for every $\rho>0$, which yields
\begin{equation}\label{mean}
 \int_{S^{n-1}}z_k(\xi)\,d\xi=0\quad\text{for every }k\in\N\,.
 \end{equation}
 
Clearly, due to \eqref{z0}, we have $z_k\equiv 0$ for $k=0,1,2$. Then we can write
$$
z(x)=\sum_{k=3}^\infty r^kz_k(\xi)
$$
Now set 
\begin{equation}\label{defv}
v(r,\xi)=\frac{r^2}{2}\sum_{i=1}^n\lambda_i\xi_i^2+\sum_{k=3}^\infty r^kz_k(\xi)\,.
\end{equation}
The assumption about convexity of $w$ then implies
\begin{equation}\label{wrr}
w_{rr}=(\sqrt{v})_{rr}=\frac{\partial}{\partial r}\left(\frac{v_r}{2\sqrt{v}}\right)=\frac{1}{4v^{3/2}}\left[2v v_{rr}-v_r^2\right]\geq 0\quad\text{for } r\geq 0\,,
\end{equation}
for every fixed direction $\xi\in S^{n-1}$. 

Now we compute $v_r$ and $v_{rr}$:
\begin{equation}\label{vr1}\begin{array}{ll}
v_r=r\sum_i\lambda_i\xi_i^2+\sum_{k=3}^\infty kr^{k-1}z_k(\xi)\,,\\
\\
v_{rr}=\sum_i\lambda_i\xi_i^2+\sum_{k=3}^\infty k(k-1)r^{k-2}z_k(\xi)\,.
\end{array}
\end{equation}
By setting
$$
A(\xi)=\frac12\sum_i\lambda_i\xi_i^2\,,
$$
we can rewrite \eqref{vr1} as follows:
\begin{equation}\label{vr2}\begin{array}{ll}
v_r=2A(\xi)r+\sum_{k=3}^\infty kr^{k-1}z_k(\xi)\,,\\
\\
v_{rr}=2A(\xi)+\sum_{k=3}^\infty k(k-1)r^{k-2}z_k(\xi)\,.
\end{array}
\end{equation}
Furthermore
$$v(r,\xi)=A(\xi)r^2+\sum_{k=3}^\infty r^kz_k(\xi)\,.
$$
Now we can compute
\begin{equation}\label{Deltaradial}
\begin{array}{rl}
&2vv_{rr}-v_r^2=4A(\xi)^2r^2+2A(\xi)\sum_3^\infty (k^2-k+2)r^kz_k(\xi)+2\left(\sum_3^\infty r^kz_k(\xi)\right)\left(\sum_3^\infty k(k-1)r^{k-2}z_k(\xi)\right)\\
\\
&-\left(2A(\xi)r+\sum_{3}^\infty kr^{k-1}z_k(\xi)\right)^2\\
\\
=&2A(\xi)\sum_3^\infty (k^2-3k+2)r^kz_k(\xi)+2\left(\sum_3^\infty r^kz_k(\xi)\right)\left(\sum_3^\infty k(k-1)r^{k-2}z_k(\xi)\right)-\left(\sum_{3}^\infty kr^{k-1}z_k(\xi)\right)^2.
\end{array}
\end{equation}

By \eqref{wrr}, we have $2vv_{rr}-v_r^2\geq 0$, we want to show that this implies $z_k\equiv 0$ for every $k\geq 3$.
We will proceed by contradiction: let $\bar{k}$ be the first index ($\geq 3$) such that $z_{\bk}$ does not identically vanish.
Then we have
$$
2vv_{rr}-v_r^2=2A(\xi)(\bk^2-3\bk+2)r^{\bk}z_{\bk}+o(r^{\bk}).
$$
On the other hand, \eqref{mean} bears the existence of  $\bar{\xi}\in S^{n-1}$ such that
$$
z_{\bk}(\bxi)<0\,,
$$
then, if $A(\bxi)>0$, for $\br$ sufficiently small we would have
$$
2v(\br,\bxi)v_{rr}(\br,\bxi)-v_r^2(\br,\bxi)=2A(\bxi)(\bk^2-3\bk+2)\br^{\bk}z_{\bk}(\bxi)+o(\br^{\bk})<0\,,
$$
which contradicts \eqref{wrr}.

If $A(\bxi)=0$, from \eqref{defv} since $v\geq 0$ we get
$$0\leq v(\br,\bxi)= \br^{\bk} z_{\bk}(\bxi) + o(\br^{\bk})<0$$
for $\br$ sufficiently small, and we have  a contradiction as before.

The proof is complete.
\end{proof}

\section{Final remarks}

We finally note that the argument of the proof above is in fact local and we can prove a more general result.

\begin{proposition}
Let $A\subseteq\rn$ be a bounded connected open set and let $u\in C^2(A)$ be such that $\Delta u=-1$ in $A$. 

If there exists $x_0\in A$ such that
$$
v(x)=u(x_0)+<\nabla u(x_0), x-x_0>-u(x)
$$
is nonnegative and $\sqrt{v}$ is convex in a neighborhood of $x_0$, then $u(x)$ is a quadratic polynomial.
\end{proposition}
\begin{proof}
As before, we can assume $x_0=0$ and that $D^2u(0)$ is diagonal. Then, with $v(x)$ as in the statement, the proof proceeds exactly as for Theorem \ref{mainthm}, starting from \eqref{eqv}.

\end{proof}

%
%
%
%
%

\bibliographystyle{amsplain}

\end{document}